\documentclass[10pt]{amsart}

\usepackage{amsfonts,amssymb,stmaryrd,amscd,amsmath,latexsym,amsbsy}

\usepackage{amssymb}
\usepackage{amsfonts}
\usepackage{latexsym}

\newtheorem{theorem}{Theorem}[section]
\newtheorem{lemma}[theorem]{Lemma}
\newtheorem{proposition}[theorem]{Proposition}
\newtheorem{corollary}[theorem]{Corollary}
\theoremstyle{definition}

\newtheorem{question}[theorem]{Question}
\newtheorem{conjecture}[theorem]{Conjecture}
\newtheorem{remark}[theorem]{Remark}

\newcommand{\C}{\mathbb C}

\begin{document}
\title[On universal Lie nilpotent associative algebras]{On
universal Lie nilpotent associative algebras}
\author[Pavel Etingof]{Pavel Etingof}
\author[John Kim]{John Kim}
\author[Xiaoguang Ma]{Xiaoguang Ma}

\address{Department of Mathematics, Massachusetts Institute of Technology,
Cambridge, MA 02139, USA}
\email{etingof@math.mit.edu}

\address{Massachusetts Institute of Technology,
Cambridge, MA 02139, USA}
\email{kimjohn@mit.edu}

\address{Department of Mathematics, Massachusetts Institute of Technology,
Cambridge, MA 02139, USA}
\email{xma@math.mit.edu}


\maketitle

\section{Introduction}

Let $A$ be an associative unital algebra over a field $k$. Let us regard
it as a Lie algebra with bracket $[a,b]=ab-ba$, and consider the
terms of its lower central series $L_i(A)$ defined inductively by
$L_1(A)=A$ and $L_{i+1}(A)=[A,L_i(A)]$. Denote by $M_i(A)$ the
two-sided ideal in $A$ generated by $L_i(A)$: $M_i(A)=AL_i(A)A$,
and let $Q_i(A)=A/M_i(A)$. 
Thus $Q_i(A)$ is the largest quotient algebra of $A$ which
satisfies the higher commutator polynomial identity
$[\ldots[[a_1,a_2],a_3],\ldots,a_i]=0$. 

An algebra $A$ is said to be Lie nilpotent of index $i$ if
$M_{i+1}(A)=0$ (i.e. $A=Q_{i+1}(A)$). For example, Lie nilpotent algebras of
index $1$ are commutative algebras. Understanding Lie nilpotent
algebras of higher indices is an interesting open problem. Many
questions about Lie nilpotent algebras can be reduced to
understanding the structure of universal Lie nilpotent algebras,
i.e. algebras $Q_{n,i}:=Q_i(A_n)$, 
where $A_n$ is the free associative algebra in $n$ generators,
since any finitely generated Lie nilpotent algebra of index $i$
is a quotient of $Q_{n,i+1}$. 

The goal of this paper is to advance our understanding of the
algebras $Q_{n,i}$ for $i\ge 2$ (in characteristic zero). The structure  
of these algebras for general $i$ and $n$ is unknown. The only
algebras $Q_{n,i}$ whose structure has been known are $Q_{n,2}$,
which is easily seen to be isomorphic to the polynomial algebra 
$k[x_1,\ldots,x_n]$, and $Q_{n,3}$, which, according to Feigin and
Shoikhet, \cite{FS}, is isomorphic to the algebra 
of even polynomial differential forms in $n$ variables, 
$k[x_1,\ldots,x_n]\otimes \wedge^{\rm even}(dx_1,\ldots,dx_n)$, 
with product $a*b=ab+da\wedge db$. 

The main result of the paper is an explicit description     
of the algebra $Q_{n,4}$. We also derive some properties of the 
algebras $Q_{n,i}$ for $i>4$, and formulate some questions for
future study which appear interesting. 

{\bf Acknowledgements.} This work arose out of the undergraduate
research project of J.K. within the framework of the UROP program at
MIT. The work of P.E. was  partially supported
by the NSF grant DMS-0504847. We are very grateful to
B. Shoikhet, who suggested to us the problem to study $Q_{n,4}$ and 
proposed an approach to it; without his help this paper
would not have appeared. We are also grateful to M. Artin for
useful discussions, and to T. Schedler for help with 
computations using MAGMA (see the proof of Theorem \ref{special}).

\section{The associated graded algebra of $Q_{n,i}$ under the Lie
filtration}

Let $A_n$ be the free
algebra over $\Bbb C$ in $n$ generators $x_1,\ldots,x_n$ ($n\ge 2$). 
The algebra $A_n$ can be viewed as the universal enveloping
algebra $U(\ell_n)$ of the free Lie algebra $\ell_n$ 
in $n$ generators. Therefore, $A_n$ has an increasing
filtration (called the Lie filtration), 
defined by the condition that $\ell_n$ sits in degree
$1$, and the associated graded algebra ${\rm gr}A_n$ under this
filtration is the commutative algebra ${\rm Sym}\ell_n$. 

The algebra $Q_{n,i}$ is the quotient of $A_n=U(\ell_n)$ by the
ideal $M_{n,i}:=M_i(A_n)$. Hence, $Q_{n,i}$ inherits the Lie
filtration from $A_n$, and one can form the quotient algebra 
$D_{n,i}={\rm gr}Q_{n,i}={\rm Sym}\ell_n/{\rm gr}M_{n,i}$, which is
commutative. 

Let $\ell_n'=[\ell_n,\ell_n]$. Then we have a natural
factorization 
$$
{\rm Sym}\ell_n=\C[x_1,\ldots,x_n]\otimes {\rm Sym}\ell_n'.
$$ 

Let $\Lambda_{n,i}$ be the image of ${\rm Sym}\ell_n'$ in $D_{n,i}$. 
Then the multiplication map 
$$
\theta: \C[x_1,\ldots,x_n]\otimes
\Lambda_{n,i}\to D_{n,i}
$$
is surjective. 

\begin{theorem}\label{fid}
\begin{enumerate}
\item[(i)]$\Lambda_{n,i}$ is a finite dimensional 
algebra with a grading by nonnegative integers (defined by
setting ${\rm deg}x_i=1$), with $\Lambda_{n,i}[0]=k$. 
\item[(ii)]The map $\theta$ is an isomorphism. 
\end{enumerate}
\end{theorem} 

\begin{proof}
Statement (i) follows from the following theorem of Jennings: 

\begin{theorem}[\cite{Jen}, Theorem 2]
If $A$ is a finitely generated Lie nilpotent algebra, 
then $M_2(A)$ is a nilpotent ideal.
\end{theorem}

This implies that there exists $N$ such that for any 
$a_1,\ldots,a_N\in M_2(A)$, \linebreak$a_1a_2\cdots a_N=0$. 
Taking $A=Q_{n,i}$, we see that for any 
$a_1,\ldots,a_N\in \ell_n'$, we have $a_1a_2\cdots a_N=0$.
Since $\Lambda_{n,i}$ is generated by the subspace $\ell_n'[<i]$ of
$\ell_n'$ of degree $<i$,
this implies that $\Lambda_{n,i}$ is finite dimensional, 
proving (i). 

To prove (ii), let $v_j, j=1,\ldots,d$, be a basis of $\Lambda_{n,i}$, and 
assume the contrary, i.e. that we have a nontrivial relation in $D_{n,i}$:  
$$
\sum_{j=1}^d f_j(x_1,\ldots,x_n)v_j=0, 
$$
where $f_j\in \Bbb C[x_1,\ldots,x_n]$. Pick this relation so that the
maximal degree $D$ of $f_j$ is smallest possible.
This degree must be positive, since $v_j$ are linearly
independent over $\Bbb C$. Applying the automorphism $g_i^t$
($t\in \Bbb C$)
of $A_n$ acting by $g_i^t(x_i)=x_i+t$, $g_i^t(x_s)=x_s$, $s\ne i$, we get 
$$
\sum_{j=1}^d f_j(x_1,\ldots,x_s+t,\ldots,x_n)v_j=0. 
$$
Differentiating this by $t$, we get 
$$
\sum_{j=1}^d \partial_{x_s} f_j(x_1,\ldots,x_n)v_j=0. 
$$
This relation must be trivial, since it has smaller degree than
$D$. 
Thus $f_j$ must be constant, which is a contradiction. 
\end{proof} 

This shows that to understand the structure of the algebra 
$Q_{n,i}$, we need to first understand the structure of the commutative
finite dimensional algebra $\Lambda_{n,i}$, which gives rise to
the following question.  

\begin{question}\label{dime} 
What is the structure of $\Lambda_{n,i}$ as a $GL(n)$-module? 
\end{question}

The answer to Question \ref{dime} has been known only for $i=2$, in which
case $\Lambda_{n,i}=\Bbb C$, 
and for $i=3$, in which case 
it is shown in \cite{FS} that $\Lambda_{n,i}=
\wedge^{\rm even}(\xi_1,\ldots,\xi_n)$, and hence 
is the sum of irreducible representations of $GL(n)$
corresponding to the partitions $(1^{2r},0,\ldots,0)$, $0\le 2r\le
n$. 

In this paper, we answer Question \ref{dime} for $i=4$. For $i>4$, the
question remains open. 

\section{The multiplicative properties of the ideals $M_i(A)$.}

A step toward understanding of the structure of the algebras
$Q_i(A)$ is understanding of the multiplicative properties 
of the ideals $M_i(A)$. In 1983, Gupta and Levin proved the
following result in this direction. 

\begin{theorem}[\cite{GL}, Theorem 3.2]\label{GLthm}
For any $m,l \geq 2$ and any algebra $A$, we have 
$$M_m(A)\cdot M_l(A)\subset M_{m+l-2}(A).$$
\end{theorem}

\begin{corollary} The space $\overline{A}:=Q_3(A)\oplus\oplus_{i\ge 3}
M_i(A)/M_{i+1}(A)$ has a structure of a graded algebra, 
with $Q_3(A)$ sitting in degree zero, and $M_i(A)/M_{i+1}(A)$ 
in degree $i-2$ for $i\ge 3$. 
\end{corollary} 

\begin{remark}
It is proved in \cite{GL} that $[[M_i(A),M_j(A)],M_k(A)]\subset
M_{i+j+k-3}(A)$, which implies that the algebra $\overline{A}$ 
is Lie nilpotent of index 2. 
\end{remark}

It is interesting that the result of Theorem \ref{GLthm} can
sometimes be improved. Namely, let us say that a pair 
$(m,l)$ of natural numbers is null if for any algebra $A$ 
$$
M_m(A)M_l(A)\subset M_{m+l-1}(A)
$$
(clearly, this property does not depend on the order of 
elements in the pair, and any pair $(1,m)$ is null). 

\begin{lemma}\label{critspec}
The pair $(m,l)$ is null if and only if the element
$$
[\ldots[x_1,x_2],\ldots, x_m]\cdot [\ldots[x_{m+1},x_{m+2}],\ldots ,x_{m+l}]
$$
is in $M_{m+l-1}(A_{m+l})$.  
\end{lemma}

\begin{proof}
By Theorem \ref{GLthm}, a pair $(m,l)$ is null iff 
$L_m(A)L_l(A)\subset M_{m+l-1}(A)$ for any $A$. 
Clearly, this happens if and only if 
the statement of Lemma \ref{critspec} holds, as desired.   
\end{proof}

\begin{theorem}\label{special} 
If $l+m\le 7$, then the unordered pair $(m,l)$ is null
iff it is not (2,2) or (2,4). 
\end{theorem}

\begin{proof}
The property of Lemma \ref{critspec} was checked using the MAGMA program, 
and it turns out that it holds for
(2,3),(3,3),(2,5),(3,4), but not for (2,2) and (2,4).

Actually, it is easy to check by hand that the property 
of Lemma \ref{critspec} does not hold for (2,2), 
and here is a computer-free proof that it holds for (2,3).  

We need to show that in $Q_{n,4}$, we have 
$$
[x_i,x_j][x_k,[x_l,x_m]]=0. 
$$
To do so, define $S(i,j,k,l,m):=[x_i,x_j][x_k,[x_l,x_m]]+
[x_k,x_j][x_i,[x_l,x_m]]$. Then in $Q_{n,4}$ we have 
$$
S(i,j,k,l,m)=0. 
$$
Indeed, it suffices to show that in $Q_{n,4}$
$$
[x_i,x_j][x_k,[x_l,x_m]]+[x_i,[x_l,x_m]][x_k,x_j]=0,
$$
which follows from the fact that in a free algebra we have 
\begin{equation*}
[a,b][c,d]+[a,d][c,b]
=[[ac,b],d] + a[d,[c,b]]
-[[a,b],d]c, 
\end{equation*}
where $a=x_i,b=x_j,c=x_k,d=[x_l,x_m]$.

Now set
\begin{eqnarray*}
&&R(i,j,k,l,m)\\
&=&-\frac{1}{2} S(x_{j},x_{k},x_{l},x_{m},x_{i}) +\frac{1}{2} S(x_{j},x_{k},x_{m},x_{l},x_{i}) -\frac{1}{2} S(x_{j},x_{k},x_{i},x_{l},x_{m})\\
&&\quad-\frac{1}{2}S(x_{j},x_{m},x_{l},x_{k},x_{i}) +\frac{1}{2}S(x_{j},x_{m},x_{k},x_{l},x_{i}) -\frac{1}{2} S(x_{j},x_{m},x_{i},x_{l},x_{k})\\
&&\quad-S(x_{j},x_{i},x_{k},x_{l},x_{m}) -S(x_{j},x_{i},x_{m},x_{l},x_{k})+\frac{1}{2} S(x_{l},x_{k},x_{m},x_{j},x_{i})\\
&&\quad -\frac{1}{2} S(x_{l},x_{k},x_{i},x_{j},x_{m}) +\frac{1}{2} S(x_{l},x_{m},x_{k},x_{j},x_{i})
-\frac{1}{2} S(x_{l},x_{m},x_{i},x_{j},x_{k})\\
&&\quad-\frac{1}{2} S(x_{k},x_{i},x_{m},x_{j},x_{l}).
\end{eqnarray*}
Then one can show by a direct computation that in $A_n$ we have 
\begin{equation*}
[x_{i},x_{j}][x_{k},[x_{l},x_{m}]]=
\frac{1}{3}(R(i, j, m, l, k)-R(i, j, l, m, k)).
\end{equation*}
Therefore, we see that $[x_{i},x_{j}][x_{k},[x_{l},x_{m}]]=0$ in
$Q_{n,4}$, as desired. 
\end{proof} 

Further computer simulations by T. Schedler using MAGMA 
have shown that the pairs (2,6) and (4,4) are not null.
This gives rise to the following conjecture: 

\begin{conjecture}
A pair $(i,j)$ is null if and only if $i$ or $j$ is odd.   
\end{conjecture}

\section{Description of $Q_{n,4}$ by generators and relations}

In \cite{FS}, Feigin and Shoikhet described the algebra 
$Q_{n,3}$ by generators and relations. Namely, 
they proved the following result. 

\begin{theorem} 
$Q_{n,3}$ is generated by $x_i$, $i=1,\ldots,n$, and 
$y_{ij}=[x_i,x_j]$, $1\le i,j\le n$, with 
defining relations 
$$
[x_i,y_{jl}]=0, 
$$
and the quadratic relation
$$
y_{ij}y_{kl}+y_{ik}y_{jl}=0
$$
saying that $y_{ij}y_{kl}$ is antisymmetric in its indices.
\end{theorem}

\begin{corollary}
The algebra $\Lambda_{n,3}$ is generated by $y_{ij}$ with
defining relations 
\begin{equation*}
y_{ij}=-y_{ji},\quad 
y_{ij}y_{kl}+y_{ik}y_{jl}=0.
\end{equation*}
\end{corollary}

In this section we would like to give a similar description of
the the algebras $Q_{n,4}$, $\Lambda_{n,4}$. 
As we know, the algebra $Q_{n,4}$ is generated
by the elements $x_i,y_{ij}$ as above, and also
$z_{ijk}=[y_{ij},x_k]$, $1\le i,j,k\le n$.
Our job is to find what relations 
to put on $x_i,y_{ij},z_{ijk}$ to generate the ideal 
$M_{n,4}$. This is done by the
following theorem, which is our main result. 

\begin{theorem}\label{main}
\begin{enumerate}
\item[(i)]The ideal $M_{n,4}$ is generated by the Lie
relations 
$$
[x_i,z_{jlm}]=0.
$$
the quadratic relations  
$$
y_{ij}z_{klm}=0, 
$$
and the cubic relations 
$$
y_{ij}y_{kl}y_{mp}+y_{ik}y_{jl}y_{mp}=0,
$$
saying that $y_{ij}y_{kl}y_{mp}$ is
antisymmetric in its indices. 
\item[(ii)] The algebra $\Lambda_{n,4}$ 
is generated by $y_{ij},z_{ijk}$ subject to the linear
relations
$$
y_{ij}=-y_{ji},\quad z_{ijk}=-z_{jik},\quad z_{ijk}+z_{jki}+z_{kij}=0, 
$$
and the relations 
$$
y_{ij}z_{klm}=0,\quad z_{ijp}z_{klm}=0,\quad y_{ij}y_{kl}y_{mp}+y_{ik}y_{jl}y_{mp}=0.
$$ 
\end{enumerate}
\end{theorem}

\begin{proof} 
Part (ii) follows from (i), so we need to prove (i). 
The relations $y_{ij}z_{klm}=0$ follow from the fact 
that $M_2(A)M_3(A)\subset M_4(A)$ for any algebra $A$ (Theorem
\ref{special}). This fact also implies that
$y_{ij}y_{kl}y_{mp}$ is antisymmetric, since 
by \cite{FS}, $y_{ij}y_{kl}+y_{ik}y_{jl}\in M_{n,3}$.

Denote by $B_n$ the quotient of $A_n$ by the relations stated in
part (i) of the theorem. We have just shown that there is a
natural surjective homomorphism $\eta: B_n\to Q_{n,4}$. 
We need to show that it is an isomorphism. For this, we need to
show that for any $a,b,c,d\in B_n$, $[[[a,b],c],d]=0$. 
For this, it suffices to show that $[[a,b],c]$ is a central
element in $B_n$. But $[[a,b],c]=0$ in $Q_{n,3}$, which implies
that $[[a,b],c]$ belongs to the ideal generated by $z_{ijk}$ and
$y_{ij}y_{kl}+y_{ik}y_{jl}$. But it is easy to see using the
relations of $B_n$ that all elements of this ideal are central in
$B_n$, as desired.    
\end{proof}

Let $K_{n,i}$ be the kernel of the 
projection map $\Lambda_{n,i+1}\to \Lambda_{n,i}$.
We see that $K_{n,3}$ is spanned by elements $z_{ijk}$
and $y_{ij}y_{kl}$ modulo the antisymmetry relation. 
Therefore, we get 

\begin{corollary}\label{34} 
As a $GL(n)$-module, $K_{n,3}$ is isomorphic to 
the direct sum of two irreducible modules 
$F_{2,1,0,\ldots,0}$ and $F_{2,2,0,\ldots,0}$
corresponding to partitions $(2,1,0,\ldots,0)$ and $(2,2,0,\ldots,0)$. 
\end{corollary}

This answers Question \ref{dime} for $i=4$. 

\begin{proof} Let $V=\Bbb C^n$ be the vector representation of
$GL(n)$. The span of $z_{ijk}$ is the subrepresentation 
of $V^{\otimes 3}$ annihilated by $Id+(12)$ and $Id+(123)+(132)$ 
in $\Bbb C[S_3]$, so it corresponds to the partition
$(2,1,0,\ldots,0)$. The span of $y_{ij}y_{kl}$ is 
the representation $S^2(\wedge^2V)/\wedge^4V$, so it is 
the irreducible representation corresponding to the partition
$(2,2,0,\ldots0)$.   
\end{proof}

\section{The $W_{n}$-module structure on $M_{n,i}/M_{n,i+1}$}

Let ${\mathfrak{g}}_n={\rm Der}(Q_{n,3})$ be
the Lie algebra of derivations of $Q_{n,3}$. 
Since every derivation of $A_n$ preserves the ideals
$M_{n,i}$, we have a natural action of 
${\rm Der}(A_n)$ on $M_{n,i}/M_{n,i+1}$ and a natural 
homomorphism $\phi: {\rm Der}(A_n)\to {\mathfrak{g}}_n$. 
This homomorphism is surjective, since a derivation 
of $A_n$ is determined by any assignment of the images of the
generators $x_i$. 

The following theorem is analogous to results of \cite{FS}. 

\begin{theorem}
The action of ${\rm Der}(A_n)$ on
$M_{n,i}/M_{n,i+1}$ factors through ${\mathfrak{g}}_n$.
Thus, ${\mathfrak{g}}_n$ acts on the graded algebra
$\overline{A}$ preserving the grading and the product.   
\end{theorem}

\begin{proof}
Let $D: A_n\to A_n$ be a derivation such that $D(A_n)\subset M_{n,3}$. 
Our job is to show that $D(M_{n,i})\subset M_{n,i+1}$ for $i\ge
1$. For this, it suffices to show that for any $a_1,\ldots,a_i\in A_n$ 
one has 
$$
D[\ldots[a_1,a_2],\ldots, a_i]\in M_{n,i+1}.
$$
For this, it is enough to prove that if $a_1,\ldots,a_i\in A_n$, and 
for some $1\le k\le i$, 
$a_k\in M_{n,3}$, then 
$$
[\ldots[a_1,a_2],\ldots,a_i]\in M_{n,i+1}.
$$
It is easy to show by induction using the Jacobi identity 
that we can rewrite $[\ldots[a_1,a_2],\ldots,a_i]$ as a linear combination of 
expressions of the form\linebreak 
$[\ldots[a_k,a_{m_1}],\ldots,a_{m_{i-1}}]$,
where $m_1,\ldots,m_{i-1}$ is a permitation of 
$1,\ldots,\hat k,\ldots,i$ ($k$ is omitted). 
Thus we may assume without loss of generality that $k=1$. 
In this case, we have to show that for any $b_1,b_2,b_3,b_4\in
A_n$, one has 
$$
[\ldots[b_1[[b_2,b_3],b_4],a_2],\ldots,a_i]\in M_{n,i+1}.
$$
This reduces to showing that for any $p,q\ge 0$ with $p+q=i-1$, 
and any $a_1,\ldots,a_p,c_1,\ldots,c_q\in A_n$, 
we have 
$$
{\rm ad}(a_1)\cdots{\rm ad}(a_p)(b_1)\cdot 
{\rm ad}(c_1)\cdots{\rm ad}(c_q){\rm ad}(b_4){\rm ad}(b_3)(b_2)\in
M_{n,i+1}.
$$
But by Theorem \ref{GLthm}, we have $M_{n,p+1}M_{n,q+3}\subset
M_{n,p+q+2}=M_{n,i+1}$, which implies the desired statement,
since the first factor is in $M_{n,p+1}$ and the second one in
$M_{n,q+3}$. 
\end{proof} 

It is pointed out in \cite{FS} that, since $Q_{n,3}$ is the
algebra of even differential forms on $\Bbb C^n$ with the
$*$-product, the Lie algebra $W_n$ of polynomial vector fields on
$\Bbb C^n$ is naturally a subalgebra of ${\mathfrak{g}}_n$. 
Therefore, we get the following corollary. 

\begin{corollary} There is a natural action of the Lie algebra
$W_n$ on the quotients $M_{n,i}/M_{n,i+1}$, and therefore 
on the graded algebra $\overline{A}$.  
\end{corollary}

It is clear from Theorem \ref{fid} that as $W_n$-modules, 
the quotients $M_{n,i}/M_{n,i+1}$ have finite length, and the
composition factors are the irreducible modules 
${\mathcal F}_D$ of tensor fields
of type $D$ (where $D$ is a Young diagram) 
considered in \cite{FS}. In fact, it follows from Theorem \ref{fid}
that if $\widetilde {\mathcal F}_D$
denotes the $W_n$-module of all polynomial tensor fields of type $D$ 
(which is reducible and therefore differs from ${\mathcal F}_D$
if and only if $D$ has only one column, which consists of $<n$ squares), and 
if 
$$
K_{n,i}=\oplus N_D F_D,
$$
where $N_D\in \Bbb Z_+$ and
$F_D$ is the irreducible representation of $GL(n)$ corresponding
to $D$, then in the Grothendieck group of the category of representations of
$W_n$ we have 
$$
M_{n,i}/M_{n,i+1}=\sum N_D{\widetilde {\mathcal F}_D}.
$$
In particular, Corollary \ref{34} implies that 
in the Grothendieck group, 
$$
M_{n,3}/M_{n,4}={ {\mathcal F}}_{2,1,0,\ldots,0}+{ {\mathcal F}}_{2,2,0,\ldots,0}.
$$

In fact, we can prove a stronger statement. 

\begin{proposition} One has an isomorphism of representations 
$$
M_{n,3}/M_{n,4}={ {\mathcal F}}_{2,1,0,\ldots,0}\oplus { {\mathcal F}}_{2,2,0,\ldots,0}.
$$ 
\end{proposition}

\begin{proof} 
Consider the subspace
$Y_n:=L_3(A_n)/(M_{n,4}\cap L_3(A_n))\subset M_{n,3}/M_{n,4}$. 
By \cite{FS}, this is a $W_n$-subrepresentation.
It is a proper subrepresentation, because $[x_1,x_2]^2\in
M_{n,3}/M_{n,4}$, but it is not contained in $Y_n$, 
as its trace in a matrix representation of $A_n$ can be nonzero. 
On the other hand, $Y_n$ contains $[x_1,[x_1,x_2]]\ne 0$, so 
$Y_n\ne 0$, and contains vectors of degree $3$. 
This easily implies that $Y_n={ {\mathcal F}}_{2,1,0,\ldots,0}$. 
On the other hand, let $Z_n$ be the subrepresentation generated 
by the elements $y_{ij}y_{kl}+y_{ik}y_{jl}$. 
These elements are annihilated by $\partial_{x_i}$, so they
generate a subrepresentation whose lowest degree is 4. 
Thus, $Z_n={ {\mathcal F}}_{2,2,0,\ldots,0}$, and 
$M_{n,3}/M_{n,4}=Y_n\oplus Z_n$, as desired. 
\end{proof} 

It would be interesting to determine the structure of the
representations $M_{n,i}/M_{n,i+1}$ when $i>3$.

\end{document}